\documentclass[12pt]{amsart}

\usepackage{fullpage}
\usepackage{amsmath}
\usepackage{amssymb}
\usepackage{amsthm}
\usepackage{amstext}
\usepackage{appendix}
\usepackage{perpage}
\usepackage[x11names]{xcolor}
\usepackage{tikz-cd}
\usepackage{ytableau}\ytableausetup{centertableaux}
\usepackage[colorlinks=true, allcolors=blue]{hyperref}
\usepackage{todonotes}

\newtheorem{thm}{Theorem}[section]
\newtheorem{lemm}[thm]{Lemma}
\newtheorem{coro}[thm]{Corollary}
\newtheorem{prop}[thm]{Proposition}
\newtheorem{thm-defi}[thm]{Theorem--Definition}

\theoremstyle{definition}
\newtheorem{expl}[thm]{Example}
\newtheorem{defi}[thm]{Definition}
\newtheorem{remark}[thm]{Remark}
\newtheorem{quest}[thm]{Question}

\begin{document}

\title{Hopf algebra and the duality operation for $\mathfrak{gl}_n(\mathbb{F}_q)$}

\author{Zhe Chen}

\address{Department of Mathematics, Shantou University, Shantou, 515063, China}
\email{zhechencz@gmail.com}

\begin{abstract}
In this paper we study the space $C(\mathfrak{gl}_n(\mathbb{F}_q))$ of complex invariant functions on $\mathfrak{gl}_n(\mathbb{F}_q)$, through a Hopf algebra viewpoint. First, we consider a variant notion of Zelevinsky's PSH algebra defined over the real numbers $\mathbb{R}$. In particular, we show that two specific $\mathbb{R}$-lattices inside the complex Hopf algebra $\bigoplus_nC(\mathfrak{gl}_n(\mathbb{F}_q))$ are real PSH algebras, and that they do not descend to $\mathbb{Z}$. Then, among consequences, we prove that every element in $C(\mathfrak{gl}_n(\mathbb{F}_q))$ is a linear combination of Harish-Chandra inductions of Kawanaka's pre-cuspidal functions, and give a conceptual characterisation of duality operation for $\mathfrak{gl}_n(\mathbb{F}_q)$, which in turn allows us to give a new proof of a classical result of Kawanaka.
\end{abstract}

\maketitle

\tableofcontents

\section{Introduction}

Let $G$ be a connected reductive group defined over a finite field $\mathbb{F}_q$, and let $\mathcal{G}$ be its Lie algebra. By definition, an invariant function on $\mathcal{G}(\mathbb{F}_q)$ is a $\mathbb{C}$-valued function on $\mathcal{G}(\mathbb{F}_q)$, that is invariant under the adjoint action of $G(\mathbb{F}_q)$. These functions, introduced in \cite{Springer1974-1975_Caracteres}, are closely related to representations of reductive groups over a finite field \cite{Lusztig_1987_fourier_Lie_alg} and over a local ring \cite{Hill_1995_Regular}, as well as to representations of Weyl groups \cite{Springer_1976_TrigSum}. 

\vspace{2mm} In this paper, we study the invariant functions on $\mathfrak{gl}_n(\mathbb{F}_q)$, simultaneously for all $n$, through a Hopf algebra viewpoint. We will focus on two aspects.

\vspace{2mm} The first aspect is motivated by Zelevinsky's PSH algebra approach to  $\bigoplus_{n}\mathbb{Z}\left[\mathrm{Irr}(\mathrm{GL}_n(\mathbb{F}_q))\right]$ (see \cite{Zelevinsky_1981_bk}), where ``PSH'' stands for ``Positive Self-adjoint Hopf''. Zelevinsky's work has been continued to classical groups in \cite{van_Leeuwen_Hopf_alg_classical_gp_1991} and \cite{Thiem--Vinroot_2007_AdvMath_finite_unitary_group}, and to unipotent groups in \cite{Gagnon_2023_Hopf}. In this paper, instead of groups, we consider a Lie algebra incarnation of $\bigoplus_{n}\mathbb{Z}\left[\mathrm{Irr}(\mathrm{GL}_n(\mathbb{F}_q))\right]$, through the PSH algebra viewpoint. 

\vspace{2mm} The second aspect of this paper is concerned with the remarkable duality operation, which was first considered by Lusztig, Curtis \cite{Curtis_duality_1980_JA}, Alvis \cite{Alvis_BAMS_1979_duality}, and Kawanaka \cite{Kawanaka_Fourier_LieAlg_FiniteField_1982}. This construction was defined for both $G(\mathbb{F}_q)$ and $\mathcal{G}(\mathbb{F}_q)$, and later also for $p$-adic groups by Aubert \cite{Aubert_1996_TAMS_Dualite}. The definition of duality operation involves the choices of a Borel subgroup and of various Levi subgroups, but it turns out that it is actually independent of these choices. The aim of this aspect is to find a more conceptual definition/characterisation of the Lie algebra duality operation.

\vspace{2mm} Two main results of this paper are:
\begin{itemize}
\item[1.]  The construction of two real PSH algebras, non-descending in the sense that they cannot be obtained via base change from $\mathbb{Z}$ (Theorem~\ref{thm: a real PSH alg} and Corollary~\ref{coro:another real PSH});
\item[2.] a choices-free definition/characterisation of the duality operations for the Lie algebras $\mathfrak{gl}_n$, simultaneously for all $n$ (Theorem--Definition~\ref{thm-defi:characterisation of duality using pre-cuspidal functions} and Proposition~\ref{prop:coincidence of the two def of duality}).
\end{itemize}
While our approach is Hopf algebra, one particular feature of the characterisation of duality operation is that its statement is free of terminology in Hopf algebra, which gives the possibility to extend the characterisation to general reductive groups and their Lie algebras; see Question~\ref{quest: general groups/Lie algebras?}. Along the way to prove these results, we also obtain some interesting by-products, to be described in the below.

\vspace{2mm} In Section~\ref{sec:prelim} we give a brief review of invariant functions, including a crucial Mackey-type formula of Letellier. Then in Proposition~\ref{prop: C(q) is Hopf} we show that 
\begin{itemize}
\item [3.] the graded sum $C(q)$ of the complex spaces of invariant functions on the finite Lie algebras $\mathfrak{gl}_n(\mathbb{F}_q)$, for all $n$, form a connected graded Hopf algebra over $\mathbb{C}$. 
\end{itemize}
This should be a known property, but its proof seems did not appear in literature in the before. Then in Section~\ref{sec:pass from Hopf to PSH} we introduce the notion of real PSH algebra, which is a direct $\mathbb{R}$-analogue of Zelevinsky's PSH algebra over $\mathbb{Z}$. A real PSH algebra is called non-descending if it cannot be realised over $\mathbb{Z}$ (see Definition~\ref{defi: descending}). From the complex Hopf algebra $C(q)$ we construct an $\mathbb{R}$-lattice, which is a non-descending real PSH algebra (the first  of the two mentioned above).

\vspace{2mm} In Section~\ref{sec:duality} we study duality operations for $\mathfrak{gl}_n(\mathbb{F}_q)$. We show that
\begin{itemize}
\item [4.] the complex invariant functions are precisely the linear combinations of the Harish-Chandra inductions of Kawanaka's pre-cuspidal functions (Theorem~\ref{thm:pre-cuspidal}), and that
\item [5.] the (unique) antipode of the Hopf algebra $C(q)$  is an alternating sum of duality operations (Proposition~\ref{prop:dual functor as antipode});
\end{itemize} 
note that the former statement can be viewed as a Lie algebra analogue to the classical result asserting that every character of $\mathrm{GL}_n(\mathbb{F}_q)$ is a uniform function (\cite[Theorem~3.2]{Lusztig_Srinivasan_char_finite_unitary_gp_1977}). These results in turn allow us to give the characterisation of duality operation mentioned above, as well as a ``finite analogue'' in Proposition~\ref{prop: finite analogue}. They also lead to
\begin{itemize}
\item [6.] a new proof of Kawanaka's classical theorem saying that the Lie algebra version of duality operation is involutive and isometric (Corollary~\ref{coro:involutivity of D_n})
\end{itemize}
based on  basic properties of Hopf algebras. In Remark~\ref{remark:transition to groups} we discuss the transition of results from the Lie algebra $\mathfrak{gl}_n(\mathbb{F}_q)$ to the  group $\mathrm{GL}_n(\mathbb{F}_q)$. For the group version of duality operation, one of its key properties is the preservation of irreducibility. However, for the Lie algebra version this preservation fails in general (by \cite[(5.3)]{Lehrer_SpInvFunc_1996}); in the end of this paper we give an application of this failure:  the construction of our second non-descending real PSH algebra mentioned above.

\vspace{2mm} Our interests in invariant functions of finite Lie algebras come from our studies on representations of Lie type groups over local rings and their relations with adjoint orbits (e.g.\ \cite{ChenStasinski_2016_algebraisation},\cite{Chen_2019_flag_orbit},\cite{ChenStasinski_2023_algebraisation_II}), to which we expect that this work will be useful as well.

\vspace{2mm} \noindent {\bf Acknowledgement.}  The author thanks Kei Yuen Chan for informing him Tadi\'c's works in the $p$-adic group setting (see \cite{Tadic+1994+Rep_classical_p-adic_gp},\cite{Tadic+1990+Crelles}). During this work, the author is partially supported by the Natural Science Foundation of Guangdong No.~2023A1515010561.

\section{Preliminaries}\label{sec:prelim}

Let $G$ be a connected reductive group over $\mathbb{F}_q$, and let $\mathcal{G}$ be its Lie algebra. We use the single notation $F$ for the geometric Frobenius endomorphisms of both $G$ and $\mathcal{G}$, which should not cause a confusion. The invariant characters of $\mathcal{G}^F$ are the complex characters of the additive group of $\mathcal{G}^F$ which are invariant with respect to the adjoint action of $G^F$. An invariant character is called irreducible, if it is not a sum of two non-zero invariant characters. Let $C(\mathcal{G}^F)$ be the complex space consisting of the complex-valued functions on $\mathcal{G}^F$ which are invariant under the adjoint action of $G^F$; this is a hermitian space with
$$(f,g)_{\mathcal{G}^F}:=\frac{1}{|G^F|}\sum_{x\in\mathcal{G}^F}f(x)\overline{g(x)},$$
and the invariant irreducible characters form an orthogonal basis (see \cite[Proposition~2.1]{Let09CharRedLieAlg}). Elements in $C(\mathcal{G}^F)$ are called invariant functions of $\mathcal{G}^F$. In the past decades, many works have been done on this space and its relation with representations of finite groups of Lie type; see e.g.\ \cite{Springer_1976_TrigSum},\cite{Springer_SteinbergFunction_1980},\cite{Lusztig_1987_fourier_Lie_alg},\cite{Lehrer_SpInvFunc_1996},\cite{Let2005book}. We are mainly interested in the case that $G=\mathrm{GL}_n$.

\vspace{2mm} Let $P$ be an $F$-stable parabolic subgroup (of $G$) with an $F$-stable Levi component $L$, and let $\mathcal{P},\mathcal{L}$ be their Lie algebras, respectively. Then there are two important operations on invariant functions: The Harish-Chandra restriction
$${^*R}_{\mathcal{L}\subseteq\mathcal{P}}^{\mathcal{G}}(f)(x)=\frac{1}{|U_P^F|}\sum_{y\in \mathcal{U}_P^F}f(x+y),$$
and the Harish-Chandra induction
$${R}_{\mathcal{L}\subseteq\mathcal{P}}^{\mathcal{G}}(f)(x)=\frac{1}{|P^F|}\sum_{\substack{g\in G^F\\ {^gx}\in \mathcal{P}^F}}f(\pi_{\mathcal{L}}({^gx})),$$
where $\pi_{\mathcal{L}}$ denotes the projection from $\mathcal{P}$ to $\mathcal{L}$. An important analogy with finite groups of Lie type is that these operations enjoy a Mackey-type intertwining formula:

\begin{prop}[\cite{Let2005paper}]\label{prop: Mackey}
Let $Q$ be an $F$-stable parabolic subgroup of $G$, and let $M$ be an $F$-stable Levi subgroup of $Q$. Denote by $\mathcal{Q}, \mathcal{M}$ their corresponding Lie algebras. Then
\begin{equation*}
{^*R}_{\mathcal{L}\subseteq\mathcal{P}}^{\mathcal{G}}
\circ 
{R}_{\mathcal{M}\subseteq\mathcal{Q}}^{\mathcal{G}}
=\sum_{s\in L^F\backslash S_G(L,M)^F  /M^F}
{R}_{\mathcal{L}\cap {^s\mathcal{M}}\subseteq \mathcal{L}\cap {^s\mathcal{Q}} }^{\mathcal{L}} 
\circ
{^*R}_{\mathcal{L}\cap {^s\mathcal{M}} \subseteq  \mathcal{P}\cap {^s\mathcal{M}}  }^{{^s\mathcal{M}}}
\circ
\mathrm{ad}(s)
\end{equation*}
where $S_G(L,M):=\{s\in G\ |\ L\cap {^sM}\ \textrm{contains a maximal torus of}\ G \}$.
\end{prop}
Actually, Letellier \cite[Proposition~3.2.9]{Let2005paper}  proves the above formula for the wider class of the Lie algebra version of Lusztig induction/restriction (for not necessarily $F$-stable $P$ and $Q$), by using techniques of two-variable Green functions. In this paper we are concerned only  with the case that ${P}$ and ${Q}$ are $F$-stable.

\vspace{2mm} It is a formal consequence of the Mackey formula that ${R}_{\mathcal{L}\subseteq\mathcal{P}}^{\mathcal{G}}$ and ${^*R}_{\mathcal{L}\subseteq\mathcal{P}}^{\mathcal{G}}$ are independent of the choice of $\mathcal{P}$ (see the argument of \cite[Theorem~5.3.1]{DM_book_2nd_edition}), so we will use the simpler notation ${R}_{\mathcal{L}}^{\mathcal{G}}$ and ${^*R}_{\mathcal{L}}^{\mathcal{G}}$.

\section{A connected graded Hopf algebra}\label{sec: a complex Hopf alg}

Consider the graded $\mathbb{C}$-module
$$C(q):=\bigoplus_{n\in\mathbb{Z}_{\geq0}} C_n(q),$$
where $C_n(q):=C(\mathfrak{gl}_n(\mathbb{F}_q))$ (here we take $\mathfrak{gl}_0(\mathbb{F}_q)=\{0\}$, so $C_0(q)=\mathbb{C}$). We shall show that this module carries a connected graded Hopf algebra structure which is both commutative and co-commutative. To do so, we first need to define multiplication/co-multiplication and unit/co-unit. 

\vspace{2mm} For convenience, we use the following simpler notation: 
\begin{itemize}
\item[-] $G_n:=\mathrm{GL}_n$ and $\mathcal{G}_n:=\mathfrak{gl}_n$;
\item[-] write $L_{k,l}$ for the block-diagonal subgroup whose first block is of size $k\times k$ and the second block is of size $l\times l$; write $\mathcal{L}_{k,l}$ for its Lie algebra;
\item[-] write ${R}_{k,l}^{n}$ and ${^*R}_{k,l}^{n}$ short for  ${R}_{\mathcal{L}_{k,l}}^{\mathcal{G}_n}$ and ${^*R}_{\mathcal{L}_{k,l}}^{\mathcal{G}_n}$, respectively, if $n=k+l$.

\end{itemize}

\begin{defi}
(i) For each pair of positive integers $k,l$, define
\begin{equation*}
m_{k,l}\colon C(\mathcal{G}_k^F)\otimes C(\mathcal{G}_l^F)\cong C(\mathcal{G}_k^F\times \mathcal{G}_l^F)\longrightarrow C(\mathcal{G}_{k+l}^F)
\end{equation*}
by the parabolic induction $R^{k+l}_{k,l}$; taking the sum of all $m_{k,l}$ gives the multiplication map $m\colon C(q)\otimes C(q)\cong C(q)$;

(ii)  Define
\begin{equation*}
m^*_{k,l}\colon C(\mathcal{G}_{k+l}^F) 
\longrightarrow C(\mathcal{G}_k^F\times \mathcal{G}_l^F)
\cong C(\mathcal{G}_k^F)\otimes C(\mathcal{G}_l^F)
\end{equation*}
by ${^*R}^{k+l}_{k,l}$. Then the co-multiplication $m^*\colon C(q)\rightarrow  C(q)\otimes C(q)$ is defined, at each component $C(\mathcal{G}_{n}^F)$, by $m^*:=\sum_{k+l=n}m^*_{k,l}$.

(iii) The unit $e\colon C_0(q)\rightarrow C(q)$ and the co-unit $e^*\colon C(q) \rightarrow C_0(q)$ are the natural embedding and the quotient, respectively. 
\end{defi}
Clearly $m,m^*,e,e^*$ are morphisms of graded modules, with the grading on $C(q)\otimes C(q)$ given by $(C(q)\otimes C(q))_n:=\bigoplus_{k+l=n}(C_k(q)\otimes C_l(q))$.

\begin{prop}\label{prop: C(q) is Hopf}
With respect to the operations $m,m^*,e,e^*$, the space $C(q)$ is a connected graded Hopf algebra over $\mathbb{C}$. Moreover,  $C(q)$ is both commutative and co-commutative. 
\end{prop}

Replacing $\mathbb{C}$-valued invariant functions by $\mathbb{R}$-valued invariant functions, the same statement of this Proposition also appears in the announcement \cite{Cuenca_Olshanski_2022_Mackey}. However, for our purpose (in Section~\ref{sec:pass from Hopf to PSH}) we cannot directly restrict to the $\mathbb{R}$-valued case, as there are too few $\mathbb{R}$-valued invariant characters.

\begin{proof}
The connectedness and the gradedness are by definition. The associativity and the co-associativity are formal consequences of the adjunction and the transitivity of ${R}_{k,l}^{n}$ and ${^*R}_{k,l}^{n}$, which are proved in \cite[Proposition~2.3.4 and Proposition~2.3.6]{Let2005paper}. The commutativity and the co-commutativity are formal consequences of the independence of ${R}_{k,l}^{n}$ and ${^*R}_{k,l}^{n}$ on the choice of parabolic subgroups.

\vspace{2mm} Since every connected graded bi-algebra admits a unique antipode (\cite[Page~238]{Sweedler1969hopf}), it remains to show that $C(q)$ is a bi-algebra, or equivalently, to show that $m^*$ is a ring morphism with respect to the multiplication on $C(q)\otimes C(q)$ given by 
$$(x\otimes y)\cdot (x'\otimes y'):= m(x\otimes x')\otimes m(y\otimes y').$$ 
In other words, we need to show that: For any $\rho_1\in C_{n_1}(q)$ and $\rho_2\in C_{n_2}(q)$, there is the identity
\begin{equation}\label{eqn:ring mor1}
m^*(m(\rho_1\otimes \rho_2))=m^*(\rho_1)\cdot m^*(\rho_2).
\end{equation}
Concerning the group version of \eqref{eqn:ring mor1} for $\mathrm{GL}_n$, there is a discussion in the appendix of \cite{Zelevinsky_1981_bk}, and one can find a more detailed (and a bit different) argument, using inner products, in \cite[Corollary~4.3.10]{Grinberg_Reiner_Notes_Hopf}. As we will see in the below, for the Lie algebra version actually the same method works.

\vspace{2mm} By definition, the left hand side of \eqref{eqn:ring mor1} is equal to
$$m^*\left(R_{n_1,n_2}^n\rho_1\boxtimes\rho_2\right)=\sum_{\substack{s+t=n,\\ s,t\in\mathbb{Z}_{\geq0}}}{^*R_{s,t}^n}\left( R_{n_1,n_2}^n\rho_1\boxtimes\rho_2 \right)$$
where $n:=n_1+n_2$, and we have
\begin{equation*}
m^*(\rho_1)=\sum_{\substack{a+b=n_1,\\ a,b\in\mathbb{Z}_{\geq0}}}{^*R_{a,b}^{n_1}}\left( \rho_1\right)
\quad
\textrm{and}
\quad
m^*(\rho_2)=\sum_{\substack{c+d=n_2,\\ c,d\in\mathbb{Z}_{\geq0}}}{^*R_{c,d}^{n_2}}\left( \rho_2\right).
\end{equation*}
So \eqref{eqn:ring mor1} is equivalent to that, for any $s,t\in\mathbb{Z}_{\geq0}$ with $s+t=n$ one has
\begin{equation}\label{eqn:ring mor2}
{^*R_{s,t}^n}\left( R_{n_1,n_2}^n\rho_1\boxtimes\rho_2 \right)=\sum_{\substack{a+b=n_1,\ c+d=n_2,\\ a+c=s,\ b+d=t,\\ a,b,c,d\in\mathbb{Z}_{\geq0}}}
{^*R_{a,b}^{n_1}}\left( \rho_1\right)\cdot{^*R_{c,d}^{n_2}}\left( \rho_2\right).
\end{equation}
By the Mackey-type formula the left hand side of \eqref{eqn:ring mor2} is equal to
\begin{equation}
\sum_{w\in L_{s,t}^F\backslash S_{G_n}(L_{s,t},L_{n_1,n_2})^F  / L_{n_1,n_2}^F}
{R}_{\mathcal{L}_{s,t}\cap {^w\mathcal{L}_{n_1,n_2}}}^{\mathcal{L}_{s,t}} 
\circ
{^*R}_{\mathcal{L}_{s,t}\cap {^w\mathcal{L}_{n_1,n_2}}}^{{^w\mathcal{L}_{n_1,n_2}}} 
\left({^w(\rho_1\boxtimes\rho_2)}\right),
\end{equation}
in which we need to analyse the index set.

\vspace{2mm} Consider
$$\mathbb{I}:=\left\{ 
\begin{bmatrix}
a & c\\
b & d
\end{bmatrix} \in M_2(\mathbb{Z}_{\geq0})
\mid
a+b=n_1,\ c+d=n_2,\ a+c=s,\ b+d=t
\right\}.$$
By \cite[Proposition~A3.2]{Zelevinsky_1981_bk} (see also \cite[Lemma~5.2.2]{DM_book_2nd_edition}) there is a natural bijection
$$\delta\colon L_{s,t}^F\backslash S_{G_n}(L_{s,t},L_{n_1,n_2})^F  / L_{n_1,n_2}^F \longrightarrow\mathbb{I}$$
determined in the way that, if
$$\delta(w)=\begin{bmatrix}
a & c\\
b & d
\end{bmatrix},$$
then $w$ admits a canonical representative (again denoted by $w$) in $G_n$:
$$w=
\begin{bmatrix}
I_a & & &  \\
& & I_c & \\
& I_b & & \\
& & & I_d\\
\end{bmatrix}\in G_n$$
(here $I_i$ denotes the identity matrix of size $i\times i$); we shall always use this representative. In particular, if $\delta(w)=\begin{bmatrix}
a & c\\
b & d
\end{bmatrix}$, then 
$$\mathcal{L}_{s,t}\cap {^w\mathcal{L}_{n_1,n_2}}=\mathcal{L}_{a,c}\times \mathcal{L}_{b,d}.$$
Thus, to prove \eqref{eqn:ring mor2} it suffices to show that, for each $\begin{bmatrix}
a & c\\
b & d
\end{bmatrix}\in\mathbb{I}$ one has
\begin{equation}\label{eqn: ringmor3}
{R}_{\mathcal{L}_{a,c}\times \mathcal{L}_{b,d}}^{\mathcal{L}_{s,t}} 
\circ
{^*R}_{\mathcal{L}_{a,c}\times \mathcal{L}_{b,d}}^{{^w\mathcal{L}_{n_1,n_2}}} 
\left({^w(\rho_1\boxtimes\rho_2)}\right)
={^*R_{a,b}^{n_1}}\left( \rho_1\right)\cdot{^*R_{c,d}^{n_2}}\left( \rho_2\right),
\end{equation}
where $w=\delta^{-1}\left(\begin{bmatrix}
a & c\\
b & d
\end{bmatrix}\right)$.

\vspace{2mm} Write
\begin{equation*}
{^*R_{a,b}^{n_1}}\left( \rho_1\right)=\sum_{\rho_1^a,\rho_1^b}\rho_1^a\boxtimes\rho_1^b
\quad
\textrm{and}
\quad
{^*R_{c,d}^{n_2}}\left( \rho_2\right)=\sum_{\rho_2^c,\rho_2^d}\rho_2^c\boxtimes\rho_2^d,
\end{equation*}
where the sums run over the possible terms (a variant of the Sweedler notation). For each quadruple $\rho_1^a,\rho_1^b,\rho_2^c,\rho_2^d$ we have
\begin{equation*}
\begin{split}
(\rho_1^a\boxtimes\rho_1^b)\cdot(\rho_2^c\boxtimes\rho_2^d)
&=R^s_{a,c}(\rho_1^a\boxtimes\rho_2^c)\otimes R^t_{b,d}(\rho_1^b\boxtimes\rho_2^d)\\
&=R^{\mathcal{L}_{s,t}}_{\mathcal{L}_{a,c}\times\mathcal{L}_{b,d}}((\rho_1^a\boxtimes\rho_2^c)\boxtimes(\rho_1^b\boxtimes\rho_2^d))\\
&=R^{\mathcal{L}_{s,t}}_{\mathcal{L}_{a,c}\times\mathcal{L}_{b,d}}(^w(\rho_1^a\boxtimes\rho_1^b\boxtimes\rho_2^c\boxtimes\rho_2^d)),
\end{split}
\end{equation*}
where the second equality follows from the definition of Harish-Chandra induction. So, to show \eqref{eqn: ringmor3} it suffices to show that
$${^*R}_{\mathcal{L}_{a,c}\times \mathcal{L}_{b,d}}^{{^w\mathcal{L}_{n_1,n_2}}} 
{^w(\rho_1\boxtimes\rho_2)}
=\sum_{\rho_1^a,\rho_1^b,\rho_2^c,\rho_2^d} {^w(\rho_1^a\boxtimes\rho_1^b\boxtimes\rho_2^c\boxtimes\rho_2^d)}.$$ 
Since
$${^{w^{-1}}\left( {^*R}_{\mathcal{L}_{a,c}\times \mathcal{L}_{b,d}}^{{^w\mathcal{L}_{n_1,n_2}}} 
{^w(\rho_1\boxtimes\rho_2)} \right)}
={^*R}_{{^{w^{-1}}(\mathcal{L}}_{a,c}\times \mathcal{L}_{b,d})}^{{\mathcal{L}_{n_1,n_2}}} 
{(\rho_1\boxtimes\rho_2)}
={^*R}_{{\mathcal{L}}_{a,b}\times \mathcal{L}_{c,d}}^{{\mathcal{L}_{n_1,n_2}}} 
{(\rho_1\boxtimes\rho_2)},$$
we only need to show
\begin{equation}\label{eqn: ringmor 4}
{^*R}_{{\mathcal{L}}_{a,b}\times \mathcal{L}_{c,d}}^{{\mathcal{L}_{n_1,n_2}}} 
{(\rho_1\boxtimes\rho_2)}=\sum_{\rho_1^a,\rho_1^b,\rho_2^c,\rho_2^d} {\rho_1^a\boxtimes\rho_1^b\boxtimes\rho_2^c\boxtimes\rho_2^d}.
\end{equation}
However, by definition we have 
$${^*R}_{{\mathcal{L}}_{a,b}\times \mathcal{L}_{c,d}}^{{\mathcal{L}_{n_1,n_2}}}(\rho_1\boxtimes\rho_2)
={^*R}_{{a,b}}^{{{n_1}}}(\rho_1)
\boxtimes 
{^*R}_{{c,d}}^{{{n_2}}}(\rho_2),$$
from which \eqref{eqn: ringmor 4} follows.
\end{proof}

\section{Real PSH algebra}\label{sec:pass from Hopf to PSH}

\begin{defi}\label{def: real PSH}
A real PSH algebra is a pair $(\mathcal{H},\Omega)$, where $\mathcal{H}=\oplus_{n\in\mathbb{Z}_{\geq0}}\mathcal{H}_n$ is a connected graded Hopf algebra defined over $\mathbb{R}$ and $\Omega$ is an $\mathbb{R}$-basis consisting of homogeneous elements, subject to the following conditions:
\begin{itemize}
\item[(P)] \emph{Positivity}. The  unit, co-unit, multiplication, and co-multiplication, restrict to maps from the subset $\mathbb{R}_{\geq0}[\Omega]$ (or $\mathbb{R}_{\geq0}[\Omega]\otimes \mathbb{R}_{\geq0}[\Omega]$) to $\mathbb{R}_{\geq0}[\Omega]$ (or $\mathbb{R}_{\geq0}[\Omega]\otimes \mathbb{R}_{\geq0}[\Omega]$).
\item[(S)] \emph{Self-adjointness}. The unit/co-unit, and the multiplication/co-multiplication, are adjoint pairs with respect to $(-,-)_{\Omega}$  and $(-,-)_{\Omega\otimes\Omega}$, where $(-,-)_{\Omega}$ and  $(-,-)_{\Omega\otimes\Omega}$ are, respectively, the (positive-definite) inner products on the spaces $\mathcal{H}$ and  $\mathcal{H}\otimes\mathcal{H}$, making the bases  $\Omega$ and $\Omega\otimes\Omega$  orthonormal.
\end{itemize} 
\end{defi}
Note that these conditions are just the PSH axioms of Zelevinsky's PSH algebra (\cite[Subsection~1.4]{Zelevinsky_1981_bk}) after formally replacing $\mathbb{Z}$ by $\mathbb{R}$ in his axioms. They reflect remarkable combinatorial properties of structure coefficients; see \cite{Grinberg_Reiner_Notes_Hopf} for more details.

\begin{prop}\label{prop:real PSH are symmetric}
Let $(\mathcal{H},\Omega)$ be a real PSH algebra, and let $\mathfrak{p}\subseteq \mathcal{H}$ be the subspace consisting of the primitive elements (i.e.\ the elements $x$ with $m^*(x)=1\otimes x+x\otimes 1$). Then $\mathcal{H}$ is isomorphic to the symmetric algebra $\mathrm{Sym}_{\mathbb{R}}\mathfrak{p}$ as Hopf algebras. In particular, $\mathcal{H}$ is both commutative and co-commutative, and it is a polynomial algebra in a basis of $\mathfrak{p}$.
\end{prop}
\begin{proof}
Let $I:=\bigoplus_{n\in\mathbb{Z}_{>0}}\mathcal{H}_n$. Then by \cite[Theorem~A1.1]{Zelevinsky_1981_bk} it suffices to show that $I=I^2\oplus \mathfrak{p}$, which follows from the same argument of \cite[Proposition~3.1.2]{Grinberg_Reiner_Notes_Hopf} (see also \cite[Theorem~3.1.7]{Grinberg_Reiner_Notes_Hopf}).
\end{proof}

Consider the pair $(R(q),\Psi(q))$ where 
$$R(q)=\bigoplus_{n\in\mathbb{Z}_{\geq0}}\mathbb{Z}\left[\mathrm{Irr}(\mathrm{GL}_n(\mathbb{F}_q))\right]$$ 
and $\Psi(q)$ consists of the irreducible characters; this is a PSH algebra (over $\mathbb{Z}$) studied in \cite{Zelevinsky_1981_bk}. Then the scalar extension $(R(q)_{\mathbb{R}},\Psi(q))$ is a real PSH algebra. This process suggests the following concept.

\begin{defi}\label{defi: descending}
Let $(\mathcal{H},\Omega)$ be a real PSH algebra. We call it descending, if there is a PSH algebra $(A,\Psi)$ over $\mathbb{Z}$ such that there is an isomorphism of connected graded Hopf algebras
$$\iota\colon A_{\mathbb{R}}\longrightarrow \mathcal{H}$$
with $\iota(\Psi)=\Omega$. Otherwise, $(\mathcal{H},\Omega)$ is called \emph{non-descending}.
\end{defi}

Clearly, this concept of (non-)descending would be interesting only if there exists non-descending real PSH algebras. In the below, for each $q$, we modify $C(q)$ to get one.

\vspace{2mm} Recall that $C(q)$ is a free $\mathbb{C}$-module with the particular basis $\overline{\Omega}(q)$ consisting of the irreducible invariant characters. By requiring homogeneous elements of different degrees to be orthogonal, the hermitian forms on the components $C_n(q)$ combine to a natural $\mathbb{C}$-valued bilinear form $(-,-)$ on $C(q)$, for which $\overline{\Omega}(q)$ is orthogonal (see \cite[Section~1.2]{Zelevinsky_1981_bk}). Normalising the elements in $\overline{\Omega}(q)$ we get an orthonormal homogeneous basis ${\Omega}(q)$ (note that this is done by multiplying the basis elements by  \emph{the quadratic roots} of some positive rational numbers). So we can consider the real inner-product space 
$$\mathfrak{C}(q):=\mathbb{R}[{\Omega}(q)],$$
graded by degrees of the elements in $\Omega(q)$. To continue, we need:

\begin{lemm}\label{lemm:Q-combination}
Let $L$ be a block-diagonal Levi subgroup of a standard parabolic subgroup of $G_n=\mathrm{GL}_n$, and let $\mathcal{L}$ be its Lie algebra. Then both $R^{\mathcal{G}_n}_{\mathcal{L}}$ and ${^*R^{\mathcal{G}_n}_{\mathcal{L}}}$ take invariant characters to $\mathbb{Q}_{\geq0}$-linear combinations of invariant characters. 
\end{lemm}

(See also the much stronger  conjecture in \cite[Conjecture~3.1]{Let09CharRedLieAlg}.)

\begin{proof}
By adjunction we only need to handle the case of $R^{\mathcal{G}_n}_{\mathcal{L}}$. 

\vspace{2mm} Let $\mathbb{F}_q[\epsilon]$ be the ring of finite dual numbers,  where $\epsilon^2=0$. Then, as a group, $\mathcal{G}_n^F$ is isomorphic to the kernel of the natural projection
$$\mathrm{GL}_n(\mathbb{F}_q[\epsilon])\longrightarrow G_n^F,$$ 
which produces a semi-direct product decomposition of subgroups
$$\mathrm{GL}_n(\mathbb{F}_q[\epsilon])\cong G_n^F\ltimes \mathcal{G}_n^F.$$
Moreover, the conjugation action of $\mathrm{GL}_n(\mathbb{F}_q[\epsilon])$ on $\mathcal{G}^F$ factors through the adjoint action of $G_n^F$ on $\mathcal{G}_n^F$. This allows one to use the usual character induction/restriction of finite groups to describe the Harish-Chandra induction $R^{\mathcal{G}_n}_{\mathcal{L}}$.

\vspace{2mm} Suppose that $P\subseteq G_n$ is a standard parabolic subgroup with Levi subgroup $L$, and write $\mathcal{P}$ for the Lie algebra of $P$. Given an invariant character $f$ of $\mathcal{L}^F$, let $\widetilde{f}$ denotes the pull-back of $f$ along the projection $\mathcal{P}^F\rightarrow \mathcal{L}^F$. Then by directly computing the formula of induced characters we see that
\begin{equation*}
R^{\mathcal{G}_n}_{\mathcal{L}}(f)
=\frac{1}{|P^F|}\cdot\frac{|\mathcal{P}^F|}{|\mathcal{G}_n^F|}\cdot
\mathrm{Res}_{\mathcal{G}_n^F}^{G_n(\mathbb{F}_q[\epsilon])}
\mathrm{Ind}_{\mathcal{P}^F}^{G_n(\mathbb{F}_q[\epsilon])}
\widetilde{f},
\end{equation*}
where  $\mathrm{Ind}$ and $\mathrm{Res}$ are the usual character induction/restriction of finite groups. 

\vspace{2mm} In particular, we see that $R^{\mathcal{G}_n}_{\mathcal{L}}(f)$ is a $\mathbb{Q}_{\geq0}$-combination of  $G_n(\mathbb{F}_q[\epsilon])$-characters restricting to the subgroup $\mathcal{G}_n^F$; however, any character of $G_n(\mathbb{F}_q[\epsilon])$ restricts to an invariant character of  $\mathcal{G}_n^F$ by construction, so the assertion holds.
\end{proof}

\begin{thm}\label{thm: a real PSH alg}
The pair $(\mathfrak{C}(q),\Omega(q))$ is a non-descending real $PSH$ algebra.
\end{thm}
\begin{proof}
By Lemma~\ref{lemm:Q-combination}, the operations $m, m^*, e, e^*$ naturally restrict to $\mathfrak{C}(q)$, so the exactly same argument of Proposition~\ref{prop: C(q) is Hopf} implies that $\mathfrak{C}(q)$ is a connected graded commutative and co-commutative Hopf algebra over $\mathbb{R}$. Now the positivity follows from again  Lemma~\ref{lemm:Q-combination}, and the self-adjointness follows from the adjointness between the Harish-Chandra induction/restriction (and that $(-,-)_{\Omega(q)}=(-,-)$ on $\mathfrak{C}(q)$, regarded as a subset of $C(q)$). So  all the conditions in Definition~\ref{def: real PSH} are fulfilled.

\vspace{2mm} For the non-descending property, it suffices to find $f,g\in\Omega(q)$ with $m(f\otimes g)\notin\mathbb{Z}[\Omega(q)]$. Since $\Omega(q)$ is orthonormal, it suffices to find $f,g,h\in\Omega(q)$ with $(m(f\otimes g),h)_{\Omega(q)}\notin\mathbb{Z}$. Consider the trivial  character $1_{\mathcal{G}_1^F}$ of $\mathcal{G}_1^F$; note that $(1_{\mathcal{G}_1^F},1_{\mathcal{G}_1^F})=\frac{q}{q-1}$, so 
$$\sqrt{\frac{q-1}{q}}\cdot1_{\mathcal{G}_1^F}\in \Omega(q).$$ 
Similarly,
$$\sqrt{\frac{q(q-1)^2(q+1)}{q^4}}\cdot1_{\mathcal{G}_2^F}\in \Omega(q).$$
Take $f=g=\sqrt{\frac{q-1}{q}}\cdot1_{\mathcal{G}_1^F}$ and $h=\sqrt{\frac{q(q-1)^2(q+1)}{q^4}}\cdot1_{\mathcal{G}_2^F}$. Then (see the notation convention in Section~\ref{sec: a complex Hopf alg})
\begin{equation*}\label{key}
	\begin{split}
		(m(f\otimes g),h)
		&=\frac{q-1}{q}\cdot { \sqrt{\frac{q(q-1)^2(q+1)}{q^4}} } \cdot  \left(  R_{1,1}^2(1_{\mathcal{L}_{1,1}^F}) ,  1_{\mathcal{G}_2^F}    \right)\\
		&=\left(\frac{q-1}{q}\right)^2\cdot { \sqrt{\frac{q+1}{q}} }
		\cdot  \left(  1_{\mathcal{L}_{1,1}^F} , {^*R}_{1,1}^2 (1_{\mathcal{G}_2^F})    \right)_{\mathcal{L}_{1,1}^F}\\
		&=\frac{1}{q^2}\cdot { \sqrt{\frac{q+1}{q}} }
		\cdot\sum_{x\in\mathcal{L}_{1,1}^F}   {^*R}_{1,1}^2 (1_{\mathcal{G}_2^F}) (x)\\
		&=\frac{1}{q^3}\cdot { \sqrt{\frac{q+1}{q}} }
		\cdot \sum_{z\in\mathcal{B}_{2}^F} 1_{\mathcal{G}_2^F}(z),
	\end{split}
\end{equation*}
where the last equality follows from the definition of Harish-Chandra restriction, in which $\mathcal{B}_2$ means the set of upper triangular matrices in $\mathcal{G}_2$. So
\begin{equation*}
	(m(f\otimes g),h)_{\Omega(q)}=(m(f\otimes g),h)
	=\sqrt{\frac{q+1}{q}}
	\notin\mathbb{Q},
\end{equation*}
as desired.
\end{proof}

Note that the above argument shows that $(\mathfrak{C}(q),\Omega(q))$ cannot even  be defined over $\mathbb{Q}$. Our second non-descending real PSH algebra will be constructed, by modifying $(\mathfrak{C}(q),\Omega(q))$, in the end of the next section.

\section{Duality operation and pre-cuspidal function}\label{sec:duality}

The notion of duality operation was first considered by Lusztig, Curtis \cite{Curtis_duality_1980_JA}, Alvis \cite{Alvis_BAMS_1979_duality}, and Kawanaka \cite{Kawanaka_Fourier_LieAlg_FiniteField_1982}, and plays an important role in the representation theory of finite groups of Lie type; later, this notion was also defined for $p$-adic groups by Aubert\cite{Aubert_1996_TAMS_Dualite}. In this section we consider the Lie algebra version of duality operation, for which several basic properties were established in \cite{Kawanaka_Fourier_LieAlg_FiniteField_1982} and \cite{Lehrer_SpInvFunc_1996}. In this section we focus on the case of $\mathfrak{gl}_n(\mathbb{F}_q)$.

\begin{defi}
The duality operation $\mathcal{D}_n$ is a linear endomorphism on the space $C_n(q)=C(\mathcal{G}_n^F)$ defined by
$$f\longmapsto \mathcal{D}_n(f):=\sum_{P} (-1)^{r(P)} R_{\mathcal{L}_P}^{\mathcal{G}_n} \circ {^* R}_{\mathcal{L}_P}^{\mathcal{G}_n}(f),$$
where the sum runs over the $F$-stable parabolic subgroups $P\subseteq G_n$ containing a fixed (arbitrary) $F$-stable Borel subgroup,  $\mathcal{L}_P$ denotes the Lie algebra of an $F$-stable Levi subgroup of $P$, and $r(-)$ means the semisimple $\mathbb{F}_q$-rank of an algebraic group. (We make the convention that $\mathcal{D}_0=\mathrm{id}$ is the identity map on $C_0(q)=\mathbb{C}$.) By basic properties of algebraic groups, $\mathcal{D}_n$ is actually independent of the choices of Borel subgroups and Levi subgroups (see e.g.\ \cite[7.2.2]{DM_book_2nd_edition}).
\end{defi}

To continue, we shall make some preparations on Hopf algebras.

\vspace{2mm} Recall that, the module $\mathrm{End}(A)$ of linear endomorphisms of a bi-algebra $A$ (over any commutative unitary ring) admits a convolution product structure, making $\mathrm{End}(A)$ an associative algebra. Then, by definition, an antipode of $A$  is a linear endomorphism that is also a two-sided inverse (with respect to the convolution product) to the identity map $\mathrm{id}_A$. It is known that every connected graded bi-algebra admits a unique antipode. We refer to \cite{Grinberg_Reiner_Notes_Hopf} for more details.

\begin{lemm}\label{lemm: use primitive element value to determine antipode}
Let $\mathcal{H}$ be a co-commutative connected graded Hopf algebra over a field $k$ of characteristic zero. Then its antipode $S$ is uniquely characterised as an  anti-endomorphism of $k$-algebra (preserving the identity element) taking $x$ to $-x$ for every primitive element $x\in \mathcal{H}$.
\end{lemm}
\begin{proof}
It is known that a co-commutative connected graded Hopf algebra over a field of characteristic zero is generated by its subspace of primitive elements, that is,
$$\mathcal{H}=k+\mathfrak{p}+\mathfrak{p}^2+\cdots,$$
where $\mathfrak{p}$ denotes the space of primitive elements of $\mathcal{H}$ (see \cite[1.5.14(d)]{Grinberg_Reiner_Notes_Hopf}). So, since for every Hopf algebra the antipode is an anti-endomorphism of algebra (see e.g.\ \cite[1.4.10]{Grinberg_Reiner_Notes_Hopf}), the values of $S$ on primitive elements would determine $S$; however, it is well-known that $S(x)=-x$ for every primitive $x$, so the assertion holds.
\end{proof}

The argument of Lemma~\ref{lemm: use primitive element value to determine antipode} also reveals a property of  pre-cuspidal invariant  functions, a notion introduced in \cite[Subsection~2.2]{Kawanaka_Fourier_LieAlg_FiniteField_1982}:

\begin{defi}\label{defi:pre-cuspidal}
Call $f\in C(\mathcal{G}^F)$ pre-cuspidal, if ${^*R}^{\mathcal{G}}_{\mathcal{L}}(f)=0$ for every $\mathcal{L}$ the Lie algebra of an $F$-stable Levi subgroup of an $F$-stable proper parabolic subgroup of $G$. 
\end{defi}

In particular, if $G$ admits no $F$-stable proper parabolic subgroup, like $G=G_1$, then every element in $C(\mathcal{G}^F)$ is pre-cuspidal.

\begin{thm}\label{thm:pre-cuspidal}
Every complex invariant function on $\mathfrak{gl}_n(\mathbb{F}_q)$ is a $\mathbb{C}$-linear combination of (possibly trivial) Harish-Chandra inductions of pre-cuspidal functions.
\end{thm}

This theorem can be viewed as an analogue to the result that every irreducible character of $\mathrm{GL}_n(\mathbb{F}_q)$ is a uniform function; see \cite[Theorem~3.2]{Lusztig_Srinivasan_char_finite_unitary_gp_1977} (see also the discussions around \cite[Theorem~11.7.3]{DM_book_2nd_edition} and \cite[Corollary~2.4.19]{Geck_Malle_2020book}).

\begin{proof}
As in the proof of Lemma~\ref{lemm: use primitive element value to determine antipode}, we have $C(q)=\mathbb{C}+\mathfrak{p}+\mathfrak{p}^2+\cdots$, where $\mathfrak{p}$ denotes the space of primitive elements in $C(q)$. Meanwhile, by definition, the pre-cuspidal invariant functions of $\mathfrak{gl}_n(\mathbb{F}_q)$ (for all $n$) are precisely the homogeneous primitive elements in the Hopf algebra $C(q)$. Since $\mathfrak{p}$ is graded (in other words,  $\mathfrak{p}=\bigoplus_n{\mathfrak{p}\cap C_{n}(q)}$; see e.g.\ \cite[1.3.19]{Grinberg_Reiner_Notes_Hopf}), the assertion follows.
\end{proof}

\begin{prop}\label{prop:dual functor as antipode}
Let $S$ be the antipode of the connected graded Hopf algebra $C(q)$. Then 
$$S=\bigoplus_{n\in\mathbb{Z}_{\geq0}}(-1)^n\cdot \mathcal{D}_n$$ 
as graded linear endomorphisms of $C(q)$. Equivalently, the duality operation $\mathcal{D}_n$ can be defined as $(-1)^n\cdot S|_{C_n(q)}$.
\end{prop}

Our proof in the below can be adapted to cover the  group version  of the same assertion in \cite[Remark~11.12]{Zelevinsky_1981_bk}, for  $R(q)=\bigoplus_{n}\mathbb{Z}\left[\mathrm{Irr}(\mathrm{GL}_n(\mathbb{F}_q))\right]$, by applying the trick of base change to proceed the argument in $R(q)_{\mathbb{Q}}$. (Note that one cannot directly work inside $R(q)$, as $R(q)$ is not generated by its submodule of primitive elements.)

\begin{proof}
Let $S':=\bigoplus_{n}(-1)^n\cdot \mathcal{D}_n$. Since $C(q)$ is commutative, by the linearity of $S'$ and Lemma~\ref{lemm: use primitive element value to determine antipode} it suffices to show that
\begin{equation}\label{temp1:alg}
S'(m(a\otimes b))=m(S'(a)\otimes S'(b))
\end{equation}
for all $a\in C_{n_1}(q)$ and $b\in C_{n_2}(q)$, and that
\begin{equation}\label{temp2:primitive}
S'(x)=-x
\end{equation}
for all primitive $x\in C(q)$. Actually, \eqref{temp1:alg} follows directly from the commutative property between duality operations and Harish-Chandra inductions (see \cite[3.15]{Lehrer_SpInvFunc_1996}), so we only need to prove \eqref{temp2:primitive}.

\vspace{2mm} Let $x=x_1+x_2\cdots$ be the homogeneous decomposition (note that the degree zero component of a primitive element in a graded Hopf algebra is always zero  \cite[1.4.17]{Grinberg_Reiner_Notes_Hopf}). Then we shall show that 
\begin{equation*}
\mathcal{D}_n(x_n)=(-1)^{n-1}x_n
\end{equation*}
for every $n\geq\mathbb{Z}_{>0}$. By definition of $C(q)$, the property $m^*(x)=1\otimes x+x\otimes 1$ is equivalent to that:
$${^*R}^{n_1+n_2}_{n_1,n_2}x_{n_1+n_2}=0$$
for every $n_1,n_2\in\mathbb{Z}_{>0}$. So the transitivity of Harish-Chandra restriction implies that
$$\mathcal{D}_n(x_n)=(-1)^{r(G_n)}{^*R}^{n}_{n}x_n=(-1)^{n-1}x_n.$$
Thus $S'=S$.
\end{proof}

In \cite{Springer_SteinbergFunction_1980}  Springer introduced the notion of Steinberg function and proved that it admits several nice properties. For $\mathfrak{gl}_n(\mathbb{F}_q)$ this function is $\mathrm{St}_n:=\mathcal{D}_n(1_{\mathcal{G}_n^F})\in C_n(q)$.

\begin{coro}
Let $S$ be the antipode of the connected graded Hopf algebra $C(q)$. Then 
$$\mathrm{St}_n=(-1)^n\cdot S(1_{\mathcal{G}_n^F})$$
for every $n$.
\end{coro}
\begin{proof}
This follows immediately from Proposition~\ref{prop:dual functor as antipode}.
\end{proof}

Now we give a conceptual characterisation of the duality operations $\mathcal{D}_n$, working simultaneously for all $n$.

\begin{thm-defi}\label{thm-defi:characterisation of duality using pre-cuspidal functions}
There exists a unique collection $\{\mathbb{D}_n\}_{n\in\mathbb{Z}_{\geq0}}$ of linear endomorphisms of the vector spaces $C(\mathcal{G}_n^F)$, one for each $n\in\mathbb{Z}_{\geq0}$, satisfying that
\begin{itemize}
\item[(i)] $\mathbb{D}_0=\mathrm{id}_{\mathbb{C}}$ on $\mathbb{C}$, and $\{\mathbb{D}_n\}_n$ commutes with Harish-Chandra inductions: 
$$\mathbb{D}_{n_1+n_2}(R_{n_1,n_2}^{n_1+n_2}(f_1\boxtimes f_2))=R_{n_1,n_2}^{n_1+n_2}(\mathbb{D}_{n_1}(f_1)\boxtimes \mathbb{D}_{n_2}(f_2))$$
for any positive integers $n_i$ and any $f_i\in C(\mathcal{G}_{n_i}^F)$ ($i=1,2$), where we view $\mathcal{G}_{n_1}\times \mathcal{G}_{n_2}$ as the block-diagonal matrices in $\mathcal{G}_{n_1+n_2}$; 
\item[(ii)] $\mathbb{D}_n(f)=(-1)^{n-1}\cdot f$ for every pre-cuspidal function $f\in C(\mathcal{G}_n^F)$, for all $n\in\mathbb{Z}_{>0}$.
\end{itemize}
\end{thm-defi}
\begin{proof}
Such a collection exists because $\{(-1)^n\cdot S|_{C_n(q)}\}_n$ is such a collection, where $S$ is the antipode of $C(q)$. 

\vspace{2mm} For the uniqueness, note that (i) implies that $S':=\bigoplus_n(-1)^n\mathbb{D}_n$ is an algebra endomorphism of the commutative algebra $C(q)$. Meanwhile, (i), (ii), and Theorem~\ref{thm:pre-cuspidal} imply that $S'$ takes every primitive element $x$ to $-x$. So Lemma~\ref{lemm: use primitive element value to determine antipode} implies that $S'$ is exactly the antipode $S$, which illustrates the uniqueness of the collection. 
\end{proof}

\begin{prop}\label{prop:coincidence of the two def of duality}
Let $\{ \mathbb{D}_n \}_n$ be the collection defined in Theorem--Definition~\ref{thm-defi:characterisation of duality using pre-cuspidal functions}. Then
$$\mathbb{D}_n=\mathcal{D}_n$$
for every $n\in\mathbb{Z}_{\geq0}$.
\end{prop}
\begin{proof}
From the argument of Theorem--Definition~\ref{thm-defi:characterisation of duality using pre-cuspidal functions} we see that $\bigoplus_n(-1)^n\mathbb{D}_n$ is the antipode of $C(q)$. So the assertion follows from Proposition~\ref{prop:dual functor as antipode}.
\end{proof}

We shall emphasise that, even though the characterisation of duality operation given in Theorem--Definition~\ref{thm-defi:characterisation of duality using pre-cuspidal functions} and Proposition~\ref{prop:coincidence of the two def of duality} is proved in terms of Hopf algebra, but its statement is free of terminology in Hopf algebra (and in particular, free of $C(q)$), so we hope to ask:

\begin{quest}\label{quest: general groups/Lie algebras?}
How to adapt the characterisation of duality operations given in Theorem--Definition~\ref{thm-defi:characterisation of duality using pre-cuspidal functions} and Proposition~\ref{prop:coincidence of the two def of duality}, so that it works for general connected reductive groups and their Lie algebras? 
\end{quest}

On the other hand, one has the following slightly stronger ``finite analogue'' of the characterisation.

\begin{prop}\label{prop: finite analogue}
Given any positive integer $m$, there exists a unique finite collection $\{\mathbb{D}_0,\mathbb{D}_1,\cdots,\mathbb{D}_m\}$ of linear endomorphisms of the vector spaces $C(\mathcal{G}_n^F)$, one for each $n\in\{0,1,\cdots,m\}$, satisfying that
\begin{itemize}
\item[(i')] $\mathbb{D}_0=\mathrm{id}_{\mathbb{C}}$ on $\mathbb{C}$, and $\{\mathbb{D}_n\}_n$ commutes with the Harish-Chandra inductions: 
$$\mathbb{D}_{n_1+n_2}(R_{n_1,n_2}^{n_1+n_2}(f_1\boxtimes f_2))=R_{n_1,n_2}^{n_1+n_2}(\mathbb{D}_{n_1}(f_1)\boxtimes \mathbb{D}_{n_2}(f_2))$$
for any positive integers $n_i$ (with $n_1+n_2\leq m$) and any $f_i\in C(\mathcal{G}_{n_i}^F)$ ($i=1,2$), where we view $\mathcal{G}_{n_1}\times \mathcal{G}_{n_2}$ as the block-diagonal matrices of $\mathcal{G}_{n_1+n_2}$; 
\item[(ii')] $\mathbb{D}_n(f)=(-1)^{n-1}\cdot f$ for every pre-cuspidal function $f\in C(\mathcal{G}_n^F)$ and $n\in\{0,1,\cdots,m\}$.
\end{itemize}
Moreover, 
$$\mathbb{D}_n=\mathcal{D}_n$$
for every $n\in\{0,1,\cdots,m\}$.
\end{prop}
\begin{proof}
It suffices to prove the uniqueness, for which we can modify the argument of Theorem--Definition~\ref{thm-defi:characterisation of duality using pre-cuspidal functions}. Let $\mathfrak{p}$ be the subspace of primitive elements in $C(q)$, and consider the subspace $\bigoplus_{n\in\{0,1,\cdots,m\}} C_n(q)$ of $C(q)$. Since $\mathfrak{p}$ is graded (see e.g.\ \cite[1.3.19]{Grinberg_Reiner_Notes_Hopf}), we have 
\begin{equation*}
\mathfrak{p}\cap \bigoplus_{n\in\{0,1,\cdots,m\}} C_n(q)=\bigoplus_{n\in\{0,1,\cdots,m\}} \mathfrak{p}_n
\end{equation*}
for $\mathfrak{p}_n:=\mathfrak{p}\cap C_n(q)$. So, by (i') and the Hopf algebra fact that $C(q)$ is generated by $\mathfrak{p}$ (\cite[1.5.14(d)]{Grinberg_Reiner_Notes_Hopf}), we see that $\mathbb{D}_n$ is determined once its values at the Harish-Chandra inductions of $\mathfrak{p}_{\lambda_1}\times\cdots\times\mathfrak{p}_{\lambda_k}$ (for all partitions $\lambda=(\lambda_1,\cdots,\lambda_k)\vdash n$) are determined. So by Theorem~\ref{thm:pre-cuspidal} (recall that elements in $\mathfrak{p}_n$ are precisely the pre-cuspidal functions in $C_n(q)$) we see that, within the conditions (i') and (ii'), the set $\{\mathbb{D}_0,\cdots\mathbb{D}_1,\cdots,\mathbb{D}_{n-1}\}$ determines $\mathbb{D}_n$. Thus the assertion follows inductively.
\end{proof}

The above inductive argument reflects the fact that the unique antipode in any connected graded Hopf algebra (over any commutative unitary ring) is determined inductively; see the argument of \cite[1.4.16]{Grinberg_Reiner_Notes_Hopf}.

\begin{expl}[The Steinberg function on $\mathfrak{gl}_2(\mathbb{F}_q)$]
It is not difficult to see that  the Steinberg function $\mathrm{St}_2$ is a non-induced neither pre-cuspidal element in $C_2(q)$.  In this example we illustrate explicitly how $\{\mathbb{D}_0,\mathbb{D}_1\}$, together with the conditions in Proposition~\ref{prop: finite analogue}, determine the value of $\mathbb{D}_2$ at  $\mathrm{St}_2$, without a priori knowing that $\mathbb{D}_2=\mathcal{D}_2$. First, by definition we have $R_{1,1}^2(1_{\mathcal{L}_{1,1}^F})=\mathrm{St}_2+1_{\mathcal{G}_2^F}$. Second, by \cite[(6)]{Springer_SteinbergFunction_1980} the difference $\mathrm{St}_2-1_{\mathcal{G}_2^F}$ is pre-cuspidal. Thus (note that every element in $C_1(q)$ is pre-cuspidal)
\begin{equation*}
\begin{split}
\mathbb{D}_2(\mathrm{St}_2)
&=\frac{1}{2}\cdot \mathbb{D}_2( R_{1,1}^2(1_{\mathcal{L}_{1,1}^F}))+\frac{1}{2}\cdot \mathbb{D}_2(\mathrm{St}_2-1_{\mathcal{G}_2^F})\\
&=\frac{1}{2}\cdot ( R_{1,1}^2(\mathbb{D}_1(1_{\mathcal{G}_1^F})\boxtimes \mathbb{D}_1(1_{\mathcal{G}_1^F}))-\frac{1}{2}\cdot (\mathrm{St}_2-1_{\mathcal{G}_2^F})\\
&=\frac{1}{2}\cdot R_{1,1}^2(1_{\mathcal{G}_2^F})-\frac{1}{2}\cdot (\mathrm{St}_2-1_{\mathcal{G}_2^F})=1_{\mathcal{G}_2^F},
\end{split}
\end{equation*}
as shall be expected.
\end{expl}

In \cite{Kawanaka_Fourier_LieAlg_FiniteField_1982} Kawanaka proved that $\mathcal{D}_n$ is involutive and isometric. Actually, the corresponding group version of this property was conjectured by Lusztig, and was settled by Curtis \cite[1.7]{Curtis_duality_1980_JA}, Alvis \cite[4.2]{Alvis_BAMS_1979_duality}, and Kawanaka \cite[2.1.2]{Kawanaka_Fourier_LieAlg_FiniteField_1982}; nowadays this property becomes a standard fact in representation theory of finite groups of Lie type. (For more details, see \cite[Page~24]{Lusztig_2021May_comments}.) Here we deduce this property via the characterisation given in Theorem--Definition~\ref{thm-defi:characterisation of duality using pre-cuspidal functions} (and Proposition~\ref{prop:coincidence of the two def of duality}), by using basic properties of commutative Hopf algebras.

\begin{coro}[\cite{Kawanaka_Fourier_LieAlg_FiniteField_1982}]\label{coro:involutivity of D_n}
Let $\{\mathbb{D}_n\}_n$ be the collection of operations defined in Theorem--Definition~\ref{thm-defi:characterisation of duality using pre-cuspidal functions}. Then $\mathbb{D}_n$ is involutive and isometric for every $n\in\mathbb{Z}_{\geq0}$. Equivalently, $\mathcal{D}_n$ is involutive and isometric for every $n\in\mathbb{Z}_{\geq0}$.
\end{coro}
\begin{proof}
As we have already seen, $\bigoplus_n(-1)^n\mathbb{D}_n$ is the antipode $S$ of $C(q)$, so it suffices to show that $S$ is involutive and isometric. Indeed, it is a basic fact of Hopf algebra that the antipode of any commutative (or co-commutative) Hopf algebra is involutive (see e.g.\ \cite[1.4.12]{Grinberg_Reiner_Notes_Hopf}). On the other hand, since $S$ is isometric when restricting to the subspace of primitive elements (Lemma~\ref{lemm: use primitive element value to determine antipode}), it is isometric by Theorem~\ref{thm:pre-cuspidal} and Theorem--Definition~\ref{thm-defi:characterisation of duality using pre-cuspidal functions}(i).
\end{proof}

\begin{remark}\label{remark:transition to groups}
After suitable modifications, all the above results (and their proofs) in this section, so far, work equally well when $C(q)$ is replaced by the scalar extension $R(q)_{\mathbb{Q}}=\bigoplus_{n}\mathbb{Q}\left[\mathrm{Irr}(\mathrm{GL}_n(\mathbb{F}_q))\right]$ of $R(q)$, as been noticed in the discussion below Proposition~\ref{prop:dual functor as antipode}. This is because all we have used are just basic properties of commutative and co-commutative connected graded Hopf algebra over a field of characteristic zero. In particular, the characterisation of duality operation (Theorem--Definition~\ref{thm-defi:characterisation of duality using pre-cuspidal functions} and Proposition~\ref{prop:coincidence of the two def of duality}), and its finite analogue (Proposition~\ref{prop: finite analogue}), adapt easily to the characters of the groups $\mathrm{GL}_n(\mathbb{F}_q)$.
\end{remark}

\begin{coro}\label{coro:another real PSH}
Let $S=\bigoplus_{n\in\mathbb{Z}_{\geq0}}(-1)^n\cdot \mathcal{D}_n$ be the antipode of $C(q)$. Then $(\mathfrak{C}(q),S(\Omega(q)))$ is a non-descending real PSH algebra, and  is different from $(\mathfrak{C}(q),\Omega(q))$ (i.e.\ $S(\Omega(q))\neq \Omega(q)$).
\end{coro}
\begin{proof}
By Corollary~\ref{coro:involutivity of D_n}, $S$ takes the homogeneous orthonormal basis $\Omega(q)\subseteq \mathfrak{C}(q)$ to another homogeneous orthonormal basis $S(\Omega(q))\subseteq \mathfrak{C}(q)$, so the commutativity between duality operations and Harish-Chandra inductions implies the first assertion. 

\vspace{2mm} To see $S(\Omega(q))\neq \Omega(q)$, since $\mathcal{D}_n(1_{\mathcal{G}_n^F})=\mathrm{St}_n$, it suffices to show that, for some large enough $n$, any non-zero $\mathbb{R}$-multiple of $\mathrm{St}_n$ is not an irreducible invariant character. Indeed, by \cite[Corollary~5.3]{Lehrer_SpInvFunc_1996}, up to a non-zero scalar $\mathrm{St}_n$ is the sum of $N_n$ irreducible invariant characters, where $N_n$ denotes the number of nilpotent  $G_n^F$-orbits in $\mathcal{G}_n^F$.  So, one can just take any $n\geq2$.
\end{proof}

\bibliographystyle{alpha}
\bibliography{zchenrefs}

\end{document}